\documentclass[11pt]{article}       
\usepackage{graphicx}
\usepackage{mathptmx}      
\usepackage{mathtools}
\usepackage{amsmath,amsthm,amssymb}
\usepackage{bbm}
\usepackage{tikz}
\usetikzlibrary{automata,arrows,positioning,calc}

\usepackage{graphicx}

\newcommand{\Rb}{\mathbbm{R}}      

\newcommand{\Bc}{\mathcal{B}}
\newcommand{\Dc}{\mathcal{D}}
\newcommand{\Qc}{\mathcal{Q}}
\newcommand{\Bb}{\mathbbm{B}}
\newcommand{\Eb}{\mathbbm{E}}
\newcommand{\Fc}{\mathcal{F}}

\newcommand{\Lc}{\mathcal{L}}
\newcommand{\Mc}{\mathcal{M}}
\newcommand{\Pc}{\mathcal{P}}

\newcommand{\Uc}{\mathcal{U}}

\newcommand{\Xc}{\mathcal{X}}
\newcommand{\Yc}{\mathcal{Y}}
\newcommand{\Wc}{\mathcal{W}}

\newcommand{\Zc}{\mathcal{Z}}
\newcommand{\1}{\mathbbm{1}}

\newcommand{\avar}{\text{AVaR}}
\newcommand{\supp}{\mathop{\rm supp}}

\newtheorem{proposition}{Proposition}[section]
\newtheorem{theorem}[proposition]{Theorem}
\newtheorem{corollary}[proposition]{Corollary}
\newtheorem{lemma}[proposition]{Lemma}
\newtheorem{definition}[proposition]{Definition}
\newtheorem{example}[proposition]{Example}

\newcommand{\prefeq}{\mathrel{\unlhd}}
\newcommand{\pref}{\mathrel{\lhd}}

\newcommand\smallmath[2]{#1{\raisebox{\dimexpr \fontdimen 22 \textfont 2
      - \fontdimen 22 \scriptfont 2 \relax}{$\scriptstyle #2$}}}

\newcommand\smallotimes{\smallmath\mathbin\otimes}

\newenvironment{tightlist}[1]{%
    \list{{\textup{(\roman{enumi})}}}{\settowidth\labelwidth{{\textup{(#1)}}}
    \leftmargin -6pt \advance\leftmargin\labelsep \itemindent \parindent
    \parsep 0pt plus 1pt minus 1pt \topsep 0pt \itemsep 0pt
    \usecounter{enumi}}}{\endlist}

\begin{document}

\title{Risk Forms: Representation,  Disintegration, and Application to Partially Observable Two-Stage Systems
}


\author{Darinka Dentcheva\footnote{Stevens Institute of Technology, Department of Mathematical Sciences, Hoboken, NJ 07030, USA}
\ and
        Andrzej Ruszczy\'nski\footnote{Rutgers University, Department of Management Science and Information Systems, Piscataway, NJ 08854, USA}  
}


%

\maketitle

\begin{abstract}
We introduce the concept of a risk form, which is a real functional of two arguments: a measurable function on a Polish space and a measure on that space. We generalize the duality theory and the Kusuoka representation to this setting. For a risk form acting on a product of Polish spaces, we define marginal and conditional forms and we prove a disintegration formula, which represents a risk form as a composition of its marginal and conditional forms. We apply the proposed approach to two-stage stochastic programming problems with partial information and decision-dependent observation distribution.

\end{abstract}

\section{Introduction}

The theory of risk measures is one of the main directions of {recent developments} in stochastic optimization. It has found multitude of applications,
 far beyond the original motivation in finance. The main setting is the following: a probability space
$(\varOmega,\Fc,P)$ is given and a space $\Zc$ of real-valued measurable functions on $\varOmega$ is defined (usually, $\Lc_p(\varOmega,\Fc,P)$ with
$p\in [1,\infty]$). A (convex) risk measure is a convex, monotonic, and translation-equivariant functional $\rho:\Zc \to \overline{\Rb}$.
We refer to  \cite{KijOhn:1993}, \cite{ogryczak1999stochastic}, \cite{ADEH:1999}, and  \cite{FolSch:2002} for initial contributions,
and to \cite{Follmer-Schied-book}, \cite{RuSh:2006a}, \cite{PflRom:07}, \cite{DPR-2009},
\cite{pflug2016multistage}  and to the survey \cite{bielecki2017survey}
for detailed presentation, applications, and further references.

Two key results provide variational representation of risk measures. One of them, called \emph{dual representation}, can be derived from the theory of conjugate duality, as shown in \cite{RuSh:2006a}. Another representation is known as  \emph{Kusuoka representation} of law invariant coherent measures of risk \cite{kus:01}. It is derived from the dual representation by employing the Hardy-Littlewood-P\'{o}lya inequality (see \cite{HLP}) under several assumptions about the properties of the measures of risk.
In all these developments, the original probability measure $P$ is assumed fixed, which is essential for the use of convex
analysis techniques in the spaces of integrable functions.

We propose a different approach. We fix a  Polish space $\Xc$ with its Borel $\sigma$-algebra $\Bc(\Xc)$, but we allow
arbitrary probability measures on this space. In Section \ref{s:risk-forms}, we introduce real-valued functionals of two arguments,
$\rho[Z,P]$, where $Z$ is a bounded measurable function on $\Xc$ and $P$ is a probability measure on $(\Xc,\Bc(\Xc))$.
In analogy to the bilinear form $\Eb[Z,P] = \int_\Xc Z(x)\,P(dx)$, we call $\rho[Z,P]$ a \emph{risk form}.
Transition risk mappings, which arose in our recent research on risk-averse control \cite{dentcheva2018time,FanRusz-2018,fan2018risk,ruszczynski2010risk}, are special cases of risk forms.

Under less restrictive assumptions
 than in the fixed probability measure case, we prove  a generalized Kusuoka representation
of risk forms in Section \ref{s:Kusuoka}. We establish the universal character of the risk representation; it remains valid for all probability measures.

The second contribution of the paper is the risk disintegration formula and its implications.
In Section \ref{s:disintegration}, 
we introduce the property of conditional consistency of risk forms.  We prove that forms enjoying this property can be represented as compositions of two forms, which we call
marginal and conditional forms. This result generalizes the decomposition of the bilinear form, resulting from the disintegration of probability measures.
While our approach is related to the theory of dynamic and conditional risk measures (see
 \cite{Scandolo:2003,Riedel:2004,roorda2005coherent,follmer2006convex,CDK:2006,RuSh:2006b,ADEHK:2007,PflRom:07,KloSch:2008,jobert2008valuations,cheridito2011composition}
 and the references therein), it allows for variable probability measures and does not have
any time structure associated with it; the order of conditioning may be arbitrary. These results are generalized in
Section \ref{s:composite}, where we consider multi-step disintegration and prove the generalized tower property of
conditional risk forms, {as a counterpart of the tower property of conditional expectations.}

Our final contribution is the application of the risk form theory to two-stage risk-averse optimization
of models with partial observation (Section \ref{s:two-stage-PO}).
Opposite to classical two-stage models, we assume that only partial information is available
at the second stage, which allows for the update of the conditional distribution of the unobserved part.
This setting was first considered in \cite{miller2011risk} and \cite{gulten2015two}, in a special case, and with
a postulated structure of the overall measure of risk. We generalize and justify the earlier contributions,
by proving the equivalence of the overall risk optimization and two-stage optimization in this setting.
We also allow for decision-dependent observation distribution and develop a risk-averse Bayes formula. In
the risk-neutral case, stochastic programming models with \emph{endogeneous} (decision-dependent) uncertainty have been discussed in \cite{jonsbraaten1998class},
 where the probability distribution and the first stage decision are linked by a special constraint.

\section{Risk Models with Variable Probability Measures}
\label{s:risk-forms}

 Let $\Pc(\Xc)$ be the set of probability measures on
 $\big(\Xc , \Bc(\Xc) \big)$.
The space  of all real-valued bounded measurable functions on $\Xc$ is denoted by $\Bb(\Xc)$. We use $x$ to denote an element of $\Xc$ and
$\delta_x$ to denote the Dirac measure concentrated at $x$. The symbol $\1$ stands for the function in $\Bb(\Xc)$,  which is constantly equal to 1.

A \emph{probabilistic model} is a pair $[Z,P] \in \Bb(\Xc) \times \Pc(\Xc)$ { Here and elsewhere, the Borel $\sigma$-algebra $\Bc(\Xc)$ is considered.}
{ For two probabilistic models  $[V,P]$ and $[W,Q]$ the notation $[V,P]{\sim} [W,Q]$ means that $P\{ V \le \eta\} = Q\{ W \le \eta\}$ for all $\eta\in \Rb$ (both models have the same distribution).}

Our goal is to propose a universal approach to risk evaluation of a family of probabilistic models.

\begin{definition}
\label{d:risk}
A measurable functional $\rho : \Bb(\Xc) \times \Pc(\Xc) \to \Rb$ is called a \emph{risk form}.
\begin{tightlist}{(iii)}
\item It is \emph{monotonic}, if $V \le W$ implies $\rho[V,P] \le \rho[W,P]$ for all $P\in \Pc(\Xc)$;
\item It is \emph{normalized} if $\rho[0,P]=0$ for all $P\in \Pc(\Xc)$;
\item It is \emph{translation equivariant} if for all $V \in \Bb(\Xc)$, all $a\in \Rb$, and all $P\in \Pc(\Xc)$, 	
$\rho[a\1+V,P]=a + \rho[V,P]$;
\item It is \emph{positively homogeneous}, if for all $V \in \Bb(\Xc)$, all $\beta\in \Rb_+$, and all $P\in \Pc(\Xc)$,
$\rho[\beta V,P] = \beta\rho[V,P]$;
\item It is \emph{law invariant}  if
	{ $[V,P] {\sim} [W,Q]$ implies that $\rho[V,P] = \rho[W,Q]$};
\item It has the \emph{support property}, if $\rho\big[\1_{\supp(P)}V,P\big] = \rho[V,P]$ for all $(V,P) \in \Bb(\Xc) \times \Pc(\Xc)$.

\end{tightlist}
\end{definition}

An example of a risk form is the expected value, which is a well-understood bilinear form:
\[
\Eb[Z,P] = \int_\Xc Z(x)\;P(dx).
\]
In our analysis, we are interested mainly in forms depending on one or both arguments in a nonlinear way.


Our concept of law invariance is broader than that used in the literature, because it allows for the probability measure
to vary. If the risk form is law invariant, then it has the support property, because $[V,P] {\sim}[\1_{\supp(P)}V,P] $.
\begin{lemma}
\label{l:state-consistency}
If a risk form $\rho : \Bb(\Xc) \times \Pc(\Xc) \to \Rb$ has the normalization, translation equivariance, and support properties then for every $Z\in \Bb(\Xc)$ and every $x\in \Xc$
\begin{equation}
\label{state-consistency}
\rho\big[Z,\delta_{x}\big] = Z(x).
\end{equation}
\end{lemma}
\begin{proof}
Using the support property twice, the translation property, and the normalization property, we obtain the chain of equations:
\[
\rho\big[Z,\delta_{x}\big] =  \rho\big[\1_{x}Z,\delta_{x}\big]
= \rho\big[Z(x)\1,\delta_{x}\big]=              Z(x) + \rho[0,\delta_{x}] = Z(x).
\]
\qed
\end{proof}
This property was called \emph{state-consistency} in \cite{dentcheva2017risk}.

Essential role in our analysis will be played by the increasing convex order { (the counterpart of the second order stochastic dominance, when
smaller outcomes are preferred)}.
\begin{definition}
\label{d:icx}
A probabilistic model $[Z,P]$ is smaller than a probabilistic model $[Z',P']$ in the \emph{increasing convex order}, written
$[Z,P] \preceq  [Z',P']$, if
for all $\eta\in \Rb$
\[
\int_\Xc \big[Z(x) - \eta\big]_+ \,P(dx) \le \int_\Xc \big[Z'(x) - \eta\big]_+ \,P'(dx).
\]
Here, $[a]_+ = \max(0,a)$.
\end{definition}
This concept allows us to consider risk forms consistent with the increasing order.
\begin{definition}
\label{d:consistency-icx}
A risk form $\rho : \Bb(\Xc) \times \Pc(\Xc) \to \Rb$ is \emph{consistent with the increasing convex order}, if
\[
[Z,P] \preceq  [Z',P']\ \Longrightarrow \ \rho[Z,P] \le \rho[Z',P'].
\]
\end{definition}
Evidently, consistency with the increasing convex order implies monotonicity and law invariance.

We call two functions $Z,V\in \Bb(\Xc)$ \emph{comonotonic}, if
\[
\big(Z(x')-Z(x)\big)\big(V(x') - V(x)\big) \ge 0,\quad \forall\,x,x'\in \Xc.
\]
\begin{definition}
A risk form $\rho : \Bb(\Xc) \times \Pc(\Xc) \to \Rb$ is \emph{comonotonically convex}, if for all comonotonic functions $Z,V\in \Bb(\Xc)$, all $P\in \Pc(\Xc)$, and all $\lambda\in [0,1]$,
\[
\rho[\lambda Z + (1-\lambda) V, P] \le \lambda \rho[Z,P]+(1-\lambda)\rho[V,P].
\]
\end{definition}

\section{Dual and Kusuoka Representations}
\label{s:Kusuoka}

In this section, we generalize the dual representation and the Kusuoka representation of law invariant coherent risk measures \cite{kus:01,RuSh:2006a} to risk forms. In the extant literature, these representations is always derived under the assumption that the probability measure is fixed (see, \emph{e.g.}, \cite{FrittelliRosazza-2005,JST-2006,PflRom:07}). We show that a more general result using variable probability measures is true.

With every probabilistic model $[Z,P]$, we associate its distribution function,
\[
F{[Z,P]}(z) = P[Z \le z], \quad z \in \Rb,
\]
and its quantile function,
\[
\varPhi{[Z,P]}(p) = \inf\big\{\eta: P[Z \le \eta] \ge p\big\}, \quad p\in (0,1].
\]
The quantile functions are elements of the space $\mathbb{Q}_{\text{b}}$ of bounded, nondecreasing, and left-continuous functions on $(0,1]$.

We first adapt the general duality result of \cite{dentcheva2014risk} for risk models on the space of quantile functions.
We denote by $\Mc$ the set of countably additive finite measures on $(0,1]$. 
For every risk form $\rho : \Bb(\Xc) \times \Pc(\Xc) \to \Rb$, we
define the conjugate functional $\rho^*:\Mc\to \Rb\cup\{+\infty\}$ as follows:
\begin{equation}
\label{conjugate}
\rho^*(\mu) = \sup_{[Z,P]\in \Bb(\Xc) \times \Pc(\Xc)} \bigg\{ \int_0^1 \varPhi[Z,P](p)\;\mu(dp) - \rho[Z,P] \bigg\}.
\end{equation}

We recall that a \emph{total preorder} $\prefeq$ on the space $\mathbb{Q}_{\text{b}}$ is a binary relation, which is reflexive, transitive and complete.
It is \emph{directed} if it satisfies the following conditions:
\begin{tightlist}{(iii)}
\item
For any real numbers $\alpha < \beta$, the relation $\alpha\1 \pref \beta \1$ is true;
\item
For every $\varPsi\in \mathbb{Q}_{\text{b}}$,  numbers $\alpha$ and $\beta$ exist such that $\alpha\1 \prefeq \varPsi \prefeq \beta\1$.
\end{tightlist}
In \cite{dentcheva2014risk}, we introduced the following properties of preorders:
\begin{description}
\item[\emph{Dual Translation:}] For all $\varPsi_1$ and $\varPsi_2$ in  $\mathbb{Q}_{\text{b}}$ and all $c\in\Rb$
\[
\varPsi_1 \prefeq \varPsi_2 \;\Longrightarrow\; \varPsi_1 + c\1 \prefeq \varPsi_2 +c\1.
\]
\item[\emph{Dual Monotonicity:}] For all $\varPsi_1$ and $\varPsi_2$ in  $\mathbb{Q}_{\text{b}}$
\[
\varPsi_1 \le \varPsi_2\ \text{pointwise} \;\Longrightarrow \; \varPsi_1  \prefeq \varPsi_2 .
\]
\end{description}

We obtain the following dual representation of risk forms.
\begin{theorem}
\label{t:dual}
{ If the space $\Xc$ is uncountable} and a risk form $\rho : \Bb(\Xc) \times \Pc(\Xc) \to \Rb$ is normalized, translation equivariant, comonotonically convex,
and consistent with the increasing convex order, then a uniquely defined closed convex set
\begin{equation}
\label{D-def}
\mathcal{D}_\rho \subseteq \big\{ \mu\in \Mc: \mu(0,\cdot]\textup{ is nondecreasing and convex on } (0,1],\ \mu(0,1] = 1\big\}
\end{equation}
exists, such that for all $[Z,P] \in  \Bb(\Xc) \times \Pc(\Xc)$
\begin{equation}
\label{conjugate-duality-averse}
\rho[Z,P] = \sup_{\mu\in\mathcal{D}_\rho}\bigg\{ \int_0^1 \varPhi[Z,P](p)\;\mu(dp) - \rho^*(\mu)\bigg\}.
\end{equation}
If, additionally, $\rho$ is positively homogeneneous, then $\rho^*(\mu)=0$ for
all $\mu\in \mathcal{D}_\rho$.
\end{theorem}
\begin{proof}
First, we show that the risk form $\rho[\cdot,\cdot]$ defines a functional $R$ on the space of quantile functions $\mathbb{Q}_{\text{b}}$ on $(0,1]$
by the identity:
\begin{equation}
\label{R-def}
 R\big( \varPhi{[Z,P]} \big) = \rho[Z,P] .
\end{equation}
Indeed, $\rho[\cdot,\cdot]$ is law invariant, due to its consistency with the increasing convex order. Therefore,
if $\varPhi{[Z,P]} = \varPhi{[Z',P']}$ then  $\rho[Z,P]=\rho[Z',P']$. Thus the functional $R$ is well-defined on the set of quantile functions $\big\{\varPhi{[Z,P]}: Z\in \Bb(\Xc),\ P\in \Pc(\Xc)\big\}$.

{ Since $\Xc$ is an uncountable Polish space, it is isomorphic to [0,1] equipped with the Borel $\sigma$-algebra (see, \emph{e.g.}, \cite[Th. 2.8 and Th. 2.12]{parthasarathy2005probability}).
Denote by $Z_0:\Xc \to [0,1]$ the said isomorphism. Since both $Z_0$ and $Z_0^{-1}$ are measurable, we can define a probability measure $P_0$ on $\Bc(\Xc)$ by $\lambda\circ Z_0$, where $\lambda$ is the Lebesgue measure on [0,1].
By construction, $P_0\big[ Z_0 \le p\big] = p$ and thus
$\varPhi{[Z_0,P_0]}(p) = p$, for all $p\in [0,1]$.
 Moreover, for every function  $\varPsi\in \mathbb{Q}_{\text{b}}$, we can define
$Z(x) = \varPsi(Z_0(x))$, $x\in \Xc$,  and  $\varPsi^{-1}(z) = \sup\, \{ p\in [0 ,1] : \varPsi(p) \le z \}$.
Then, for every $z \in \Rb$,
\begin{equation}
\label{Psi-inverse}
P_0\big\{x: Z(x) \le z\big\} = P_0\big\{x: \varPsi(Z_0(x)) \le z\big\} =
P_0\big\{x: Z_0(x) \le \varPsi^{-1}(z)\big\} = \varPsi^{-1}(z).
\end{equation}
Consequently, the distribution function of $Z$ under $P_0$ is the inverse of $\varPsi$, and thus $\varPhi{[Z,P_0]} = \varPsi$.
This means that the domain of $R$ is the entire space $\mathbb{Q}_{\text{b}}$.}

We verify the assumptions of Theorem 4 of \cite{dentcheva2014risk}. We define a preference relation $\prefeq$ on $\mathbb{Q}_{\text{b}}$ by setting
\[
\varPsi_1 \prefeq \varPsi_2\; \text{ if and only if }\; R(\varPsi_1 ) \leq R(\varPsi_2).
\]
Clearly, the relation $\prefeq$ is a total preorder with $R$ being its numerical representation.
Since $\rho[\cdot,\cdot]$ is normalized and translation equivariant,
the identity \eqref{R-def} implies that $R$ is normalized and translation equivariant as well.
We observe that $R$ is monotonic, i.e., if $\varPsi_1 \leq \varPsi_2$ (pointwise), then $R(\varPsi_1 ) \leq R(\varPsi_2)$.
Indeed, let  $\varPsi_1 \leq \varPsi_2$  and set $Z_1(x)= \varPsi_1 (Z_0(x)) $,  $Z_2(x)= \varPsi_2 (Z_0(x)) $ for all $x\in \Xc$.
Similar to \eqref{Psi-inverse},
\[
P_0 ( Z_1\leq z) = \varPsi_1^{-1}(z) \ge  \varPsi_2^{-1}(z) = P_0 (Z_2\le z)
\]
 for all $z\in\Rb$. The last relation implies that $[Z_1,P_0] \preceq [Z_2,P_0]$.
 The consistency of $\rho$ with the increasing convex order entails
\[
R\big( \varPsi_1 \big) = \rho[Z_1,P_0] \le \rho[Z_2,P_0] = R\big( \varPsi_2 \big),
\]
which is the desired monotonicity.
The properties of $R$ further imply that the order $\prefeq$ is directed, monotonic, and satisfies the dual translation property.

For any two comonotonic functions $Z_1$ and $Z_2$ in $\Bb(\Xc)$, any $\lambda\in [0,1]$,
and any $P\in \Pc(\Xc)$,
\[
\varPhi\big[\lambda Z_1 + (1-\lambda)Z_2, P \big] = \lambda \varPhi[Z_1,P] + (1-\lambda) \varPhi[Z_2,P].
\]
The comonotonic convexity assumption implies that
\begin{multline*}
R\big( \lambda \varPhi[Z_1,P] + (1-\lambda) \varPhi[Z_2,P] \big) =  \rho\big[ \lambda Z_1 + (1-\lambda)Z_2,P\big] \\
\le \lambda \rho[Z_1,P] + (1-\lambda) \rho[Z_2,P] = \lambda R\big( \varPhi[Z_1,P]\big) + (1-\lambda)  R\big( \varPhi[Z_2,P]\big).
\end{multline*}
Since any two functions $\varPsi_1,\varPsi_2\in \mathbb{Q}_{\text{b}}$ can be represented as
\begin{align*}
\varPsi_1 &= \varPhi\big[Z_1,P_0\big],\quad \text{with}\quad Z_1(x) = \varPsi_1(Z_0(x)), \ x\in \Xc,\\
\varPsi_2 &= \varPhi\big[Z_2,P_0\big],\quad \text{with}\quad Z_2(x) = \varPsi_2(Z_0(x)), \ x\in \Xc,
\end{align*}
and the functions $Z_1$ and $Z_2$ are comonotonic by construction, the functional $R$
is convex.

Consider a function $\varPsi\in \mathbb{Q}_{\text{b}}$. For $(a,b]\subset (0,1]$ we define
\begin{equation}
\label{quantile-cond}
\varPsi_{(a,b]}(p) = \begin{cases}
\displaystyle{\frac{1}{b-a} \int_a^b \varPsi(s)\;ds} & \text{if } p\in (a,b], \\
\varPsi(p) & \text{otherwise}.
\end{cases}
\end{equation}
Directly from \eqref{quantile-cond} we observe that for every $\alpha \in (0,1]$
\[
\int_{1-\alpha}^1 \varPsi_{(a,b]}(p)\;dp  \le \int_{1-\alpha}^1 \varPsi(p)\; dp.
\]
Therefore, for any $[Z,P]$ and $[V,Q]$ such that $\varPsi_{(a,b]} =  \varPhi[Z,P]$ and
$\varPsi = \varPhi[V,Q]$, we have $[Z,P] \preceq  [V,Q]$. Due to the consistency of $\rho[\cdot,\cdot]$
with the increasing convex order,
\[
R\big( \varPsi_{(a,b]} \big) = \rho[Z,P] \le \rho[V,Q] = R\big( \varPsi \big).
\]
Therefore, the preorder $\prefeq$ is risk averse in the sense of \cite[Def. 2]{dentcheva2014risk}.
It follows from \cite[Th. 4]{dentcheva2014risk} that a set $\Dc_{\rho}$ satisfying \eqref{D-def} exists,
such that
\begin{align*}
R(\varPsi) &= \sup_{\mu\in\mathcal{D}_{\rho}}\bigg\{ \int_0^1 \varPsi(p)\;\mu(dp) - R^*(\mu)\bigg\},\\
\intertext{with}
R^*(\mu)&=\sup_{\varPsi\in \mathbb{Q}_{\text{b}}} \bigg\{ \int_0^1 \varPsi(p)\;\mu(dp) - R(\varPsi) \bigg\}.
\end{align*}
Moreover, $R^*(\mu) = 0 $ for $\mu \in \Dc_{\rho}$, if $R$ is positively homogeneous. The assertion of the theorem
follows now from
 the substitution \eqref{R-def}.
 \qed
\end{proof}

{
The dual representation \eqref{conjugate-duality-averse} can be written with the use of the Stjelties integral with respect to the
distribution function $w(\cdot)= \mu(0,\cdot]$. It is particularly revealing in the homogeneous case:
\begin{equation}
\label{conjugate-duality-averse-2}
\rho[Z,P] = \sup_{w\in\Wc_\rho} \int_0^1 \varPhi[Z,P](p)\;dw(p),
\end{equation}
where each $w(\cdot)\in \Wc_{\rho}$ can be interpreted as a dual (rank dependent) utility function. The set $\Wc_{\rho}$ is
the set of distribution functions of measures $\mu\in \Dc_{\rho}$: a convex subset of convex and nondecreasing functions from $[0,1]$ to $[0,1]$.
The relation \eqref{conjugate-duality-averse-2} suggests an intriguing relation of law invariant risk forms and the dual utility theory
of \cite{Quiggin:1982} and \cite{Yaari:1987}, as analyzed in \cite{dentcheva2013common}.
}

Theorem \ref{t:dual} allows us to derive a generalization of the celebrated Kusuoka
representation of law invariant coherent measures of risk
(see \cite{kus:01}, \cite{FrittelliRosazza-2005}, and \cite[sec. 2.2.4]{PflRom:07} for an overview of relevant results).
\begin{definition}
The Average Value at Risk at level $\alpha\in [0,1]$ of a probabilistic model $[Z,P]$ is defined as follows:
\[
\textup{AVaR}_\alpha[Z,P] = \begin{cases}
\frac{1}{\alpha}\int_{1-\alpha}^1 \varPhi[Z,P](p)\;dp &\ \text{if } \alpha \in (0,1),\\
\varPhi[Z,P](1) &\ \text{if } \alpha = 0,\\
\Eb[Z,P]  &\ \text{if } \alpha = 1.
\end{cases}
\]
\end{definition}
Then, repeating the considerations leading to \cite[Cor. 1]{dentcheva2014risk} \emph{verbatim}, we obtain
the following result.

\begin{corollary}
\label{c:Kusuoka}
Suppose the conditions of Theorem \ref{t:dual} are satisfied and the risk form $\rho[\cdot,\cdot]$ is positively homogeneous. Then
a convex subset { $\varLambda_\rho$} of the set of probability measures on $[0,1]$ exists, such that
for all $[Z,P]\in \Bb(\Xc)\times \Pc(\Xc)$
\begin{equation}
\label{Kusuoka-full}
\rho[Z,P] = \sup_{\lambda\in \varLambda_\rho} \int_0^1 \textup{AVaR}_s[Z,P] \;\lambda(ds).
\end{equation}
\end{corollary}
{ It it worth stressing that in the extant literature, the Kusuoka representation was derived for probabilistic models
with a fixed atomless probability measure $P$. If $P$ has atoms, additional conditions are needed, as discussed in \cite{noyan2015kusuoka}.
Our approach proves the validity of the Kusuoka representation
for probabilistic models with an arbitrary probability measure $P$. The set $\Lambda_\rho$ is determined by the risk form alone; it is
the same for all probabilistic models $[Z,P]$.
This universal property is due to two key conditions: the consistency with the increasing convex order with \emph{both} $Z$ and $P$ varying, and the requirement
that $\Xc$ be Polish and uncountable. }

\section{The risk disintegration formula}
\label{s:disintegration}

Our interest in this section is measuring risk on product spaces.
Consider two {Polish} spaces $\Xc$ and $\Yc$ and their corresponding Borel $\sigma$-algebras  $\Bc(\Xc)$ and $\Bc(\Yc)$.
We can disintegrate any $P\in \Pc(\Xc\times\Yc)$ into its marginal $P_{\Xc}\in \Pc(\Xc)$ and a transition kernel $P_{\Yc|\Xc}:\Xc\to \Pc(\Yc)$
as follows: $P(dx, dy) =  P_{\Xc} (dx) \, P_{\Yc|\Xc}(dy|x)$.

Let $\Qc(\Yc|\Xc)$ be the space of all measurable mappings $Q:\Xc\to \Pc(\Yc)$ (transition kernels). For any measure $\lambda\in \Pc(\Xc)$ and
 any kernel $Q\in \Qc(\Yc|\Xc)$, { the measure $\lambda \smallotimes Q $, defined as
 \[
 \big[\lambda \smallotimes Q\big](dx,dy) = \lambda(dx)\, Q(dy|x),
 \]}
  is an  element of $\Pc(\Xc \times \Yc)$.

Suppose the risk form $\rho:\Bb(\Xc\times\Yc)\times\Pc(\Xc\times\Yc)\to \Rb$ is monotonic, translation equivariant, and normalized.
 Then it induces a mapping $\rho_{\Yc|\Xc}:\Bb(\Xc\times\Yc)\times\Qc(\Yc|\Xc)\to \Bb(\Xc)$
defined as follows:
\begin{equation}
\label{cond-operator}
\rho_{\Yc|\Xc}[Z,Q](x) = \rho[Z, \delta_x \smallotimes Q],\quad x \in \Xc.
\end{equation}
We call the mapping $\rho_{\Yc|\Xc}[\cdot,\cdot]$ the \emph{conditional risk operator} associated with $\rho[\cdot,\cdot]$.

To verify that the values are indeed elements of $\Bb(\Xc)$,
let   $c\in \Rb$ be such that $Z \le c\1$. Then, by monotonicity, translation equivariance and normalization,
\[
\rho[ Z, \delta_x\smallotimes Q] \le \rho[ c\1, \delta_x\smallotimes Q] = c.
\]
The lower bound is similar, and thus the function $\rho_{\Yc|\Xc}[Z,Q]$ is bounded.
The function $\rho_{\Yc|\Xc}[Z,Q](\cdot)$ is measurable, as a composition of measurable mappings.

If the risk form $\rho$ has the support property, then for each $x\in \Xc$ the value of \eqref{cond-operator} depends only on
the function $Z(x,\cdot)\in \Bb(\Yc) $ and the measure $Q(x) \in \Pc(\Yc)$.
We can, therefore, define the
functionals $\rho_{\Yc|x}:\Bb(\Yc)\times\Pc(\Yc)\to\Rb$,
 $x\in \Xc$, as follows:
\begin{equation}
\label{rhox}
\rho_{\Yc|x}[Z(x,\cdot),Q(x)] = \rho_{\Yc|\Xc}[Z,Q](x),\quad x \in \Xc.
\end{equation}
We call them \emph{conditional risk forms} associated with $\rho[\cdot,\cdot]$. Observe that any function from $\Bb(\Yc)$ and
any measure from $\Pc(\Yc)$ may feature as arguments of $\rho_{\Yc|x}[\cdot,\cdot]$.

From now on, we always assume that the risk forms in question have the support property. The inequalities ``$\le$'' between functions are always understood point-wise.

\begin{lemma}
\label{l:conditionals-properties}
If the risk form $\rho[\cdot,\cdot]$ is monotonic (normalized, translation equivariant), then, for every $x\in \Xc$, the
conditional risk form $\rho_{\Yc|x}$ is monotonic (normalized, translation equivariant).
\end{lemma}
\begin{proof}
All the properties follow directly from the equation
\[
\rho_{\Yc|x}[Z(x,\cdot),Q(x)]=\rho[Z, \delta_x \smallotimes Q],
\]
 which defines the conditional risk form.
 \qed
\end{proof}
\begin{definition}
\label{d:internal}
A risk form $\rho:\Bb(\Xc\times\Yc)\times\Pc(\Xc\times\Yc)\to \Rb$ is \emph{conditionally consistent} if for all
$Z,Z' \in \Bc(\Xc\times\Yc)$ and all $Q,Q' \in \Qc(\Yc|\Xc)$ such that
\[
\rho_{\Yc|\Xc}[Z,Q] \le \rho_{\Yc|\Xc}[Z',Q'],
\]
we also have
\[
\rho[ Z, \lambda\smallotimes Q] \le \rho[ Z', \lambda\smallotimes Q'], \quad \forall \, \lambda \in \Pc(\Xc).
\]
\end{definition}
The following result is the foundation of our further considerations.
\begin{theorem}
\label{t:total-risk}
Suppose a risk form $\rho:\Bb(\Xc\times\Yc)\times\Pc(\Xc\times\Yc)\to \Rb$ is monotonic, normalized, translation equivariant, has the support property, and is {conditionally consistent}. Then
a risk form $\rho_{\Xc}:\Bb(\Xc) \times \Pc(\Xc)\to \Rb$ exists, such that for all $[Z,P]\in \Bb(\Xc\times\Yc)\times\Pc(\Xc\times\Yc)$
the following formula is true:
\begin{equation}
\label{total-risk}
\rho[Z,P] = \rho_{\Xc}\big[\rho_{\Yc|\Xc}[Z,P_{\Yc|\Xc}], P_{\Xc}\big].
\end{equation}
The risk form $\rho_{\Xc}$ is uniquely defined by the equation
\begin{equation}
\label{r-def}
\rho_{\Xc}[f,P_{\Xc}] = \rho[\underline{f},P], \quad \text{with} \quad \underline{f}(x,y) \equiv f(x).
\end{equation}
It is monotonic, normalized, translation equivariant, and has the support property.
\end{theorem}
\begin{proof}
Let us verify \eqref{total-risk}. Suppose $[Z,P]$ and $[Z',P']$ are such that
\[
\rho_{\Yc|\Xc}[ Z, P_{\Yc|\Xc}] =  \rho_{\Yc|\Xc}[ Z', P'_{\Yc|\Xc}].
\]
Then it follows
from Definition \ref{d:internal} that $\rho[ Z, \lambda\smallotimes P_{\Yc|\Xc}] = \rho[ Z', \lambda\smallotimes P'_{\Yc|\Xc}]$ for all
$\lambda\in \Pc(\Xc)$. If, additionally, the marginal measures $P_{\Xc}$ and $P'_{\Xc}$ are identical, by setting $\lambda=P_{\Xc}=P'_{\Xc}$ we conclude that
\[
\rho[Z,P] = \rho[ Z, P_{\Xc}\smallotimes P_{\Yc|\Xc}] = \rho[ Z', P'_{\Xc}\smallotimes P'_{\Yc|\Xc}]=\rho[Z',P'].
\]
It follows that the value of $\rho[Z,P]$ is fully determined by the value of the conditional risk operator
$\rho_{\Yc|\Xc}[ Z, P_{\Yc|\Xc}]$ and
the marginal measure $P_{\Xc}$. Therefore, the disintegration formula \eqref{total-risk} is true.

It remains to verify the properties of $\rho_{\Xc}$. Set $Z(x,y) = \underline{f}(x,y) \equiv f(x)$ in \eqref{total-risk}. Then, by the support, translation equivariance, and
normalization properties of $\rho[\cdot,\cdot]$, for every $x \in \Xc$ we obtain
\[
\rho_{\Yc|\Xc}[ \underline{f}, P_{\Yc|\Xc}](x) = \rho[ f(x)\1, \delta_x\smallotimes P_{\Yc|\Xc}] = f(x).
\]
Combining with \eqref{total-risk}, we observe that the identity \eqref{r-def} is true.
All the postulated properties of $\rho_{\Xc}[\cdot,\cdot]$ follow  from the corresponding properties of $\rho[\cdot,\cdot]$. \qed
\end{proof}

We call the identity \eqref{total-risk} the \emph{risk disintegration formula}. It represents the risk form $\rho[\cdot,\cdot]$
by its marginal risk form $\rho_{\Xc}[\cdot,\cdot]$ and its conditional risk operator $\rho_{\Yc|\Xc}[\cdot,\cdot]$.

{
\begin{example}
\label{e:composite}
Consider the risk form $\rho:\Bb(\Xc\times\Yc)\times \Pc(\Xc\times\Yc)\to\Rb$ defined as follows:
\[
\rho[Z,P] = \min_{\eta(\cdot)} \Eb_P \Big[
\eta(x) + \frac{1}{\alpha} \big(Z(x,y) - \eta(x)\big)_+\Big],
\]
where the minimization is over all measurable mappings $\eta:\Xc\to\Rb$. Directly from Definition \ref{d:risk}
we verify that it is normalized, monotonic, translation invariant, and has the support property. To verify
the conditional consistency, observe that
\begin{equation}
\label{example-cond}
\rho_{\Yc|x}[Z,Q] = \min_{\eta}  \Big[
\eta + \frac{1}{\alpha}\Eb_{Q(x)}\big[ \big(Z(x,y) - \eta\big)_+\big]\Big].
\end{equation}
Therefore, the relation
\[
\rho[Z,\delta_x \smallotimes Q] \le \rho[Z',\delta_x \smallotimes Q'], \quad \forall\;x\in \Xc,
\]
 means that
\[
\min_{\eta}  \Big[
\eta + \frac{1}{\alpha}\Eb_{Q(x)}\big[ \big(Z(x,y) - \eta\big)_+\big]\Big] \le
\min_{\eta}  \Big[
\eta + \frac{1}{\alpha}\Eb_{Q'(x)} \big[\big(Z'(x,y) - \eta\big)_+\big]\Big],
\]
for all $x\in \Xc$. By the interchangeability of $\min_{\eta(\cdot)}$ and $\Eb_\lambda$, for every $\lambda \in \Pc(\Xc)$ we obtain:
\begin{align*}
\rho[Z,\lambda \smallotimes Q] &= \min_{\eta(\cdot)} \Eb_\lambda \Big[
\eta(x) + \frac{1}{\alpha}\Eb_{Q(x)} \big[\big(Z(x,y) - \eta(x)\big)_+\big]\Big] \\
&=  \Eb_\lambda \min_{\eta}\Big[
\eta + \frac{1}{\alpha}\Eb_{Q(x)} \big[\big(Z(x,y) - \eta\big)_+\big]\Big] \\
&\le
\Eb_\lambda \min_{\eta}\Big[
\eta + \frac{1}{\alpha}\Eb_{Q'(x)} \big[\big(Z'(x,y) - \eta\big)_+\big]\Big] = \rho[Z',\lambda \smallotimes Q'].
\end{align*}
Therefore, the assumptions of Theorem \ref{t:total-risk} are satisfied, and the risk form $\rho[\cdot,\cdot]$ can be disintegrated
into a marginal and conditional form. Using the dual representation of the
Average Value at Risk in \eqref{example-cond}, we obtain the explicit formula:
\[
\rho[Z,P] = \Eb_{P_{\Xc}}\Big[ \textup{AVaR}_\alpha\big[Z(x,y),P_{\Yc|\Xc} \big]  \Big].
\]
This means that $\rho_{\Xc}[\,\cdot\,,P_{\Xc}]=\Eb_{P_{\Xc}}[\,\cdot\,]$ and $\rho_{\Yc|\Xc}[\cdot,P_{\Yc|\Xc}]
= \textup{AVaR}_\alpha[\,\cdot\, ,P_{\Yc|\Xc}] $.

The risk form $\rho[\cdot,\cdot]$ is not law invariant on the product space, because pairs $[Z,P]$ having identical distribution
may have different conditional and marginal distributions. This can be seen on the case of $\Xc=\Yc= [0,1]$ with the Lebesgue measure $P$
on $\Xc\times\Yc$ and two functions: $Z(x,y)=x$ and $Z'(x,y)=y$. Then $[Z,P]\sim [Z',P]$. However,
 \begin{align*}
\rho[Z,P] &= \min_{\eta(\cdot)} \int_0^1 \int_0^1 \Big[
\eta(x) + \frac{1}{\alpha} \big(x - \eta(x)\big)_+\Big]\;dy\;dx \\
&= \int_0^1 \min_{\eta}\Big[
\eta + \frac{1}{\alpha} \big(x - \eta\big)_+\Big]\;dx = \int_0^1 x \;dx= \frac{1}{2} ,
\end{align*}
and
\begin{align*}
\rho[Z',P]
&= \min_{\eta(\cdot)} \int_0^1 \int_0^1 \Big[ \eta(x) + \frac{1}{\alpha} \big(y - \eta(x)\big)_+\Big]\;dy\;dx \\
&\ge \int_0^1 \min_{\eta}  \int_0^1 \Big[ \eta + \frac{1}{\alpha} \big(y - \eta\big)_+\Big]\;dy\;dx =  \int_0^1 \avar_{\alpha}(Y)\;dx = \avar_{\alpha}(Y).
\end{align*}
We observe that the minimizing $\eta$ in the second line of the formula above is in fact constant and equal to $1-\alpha$. Therefore,
it can be used in the first line, the inequality becomes an equation, and $\rho[Z',P] = 1 -\alpha/2$. We conclude that $\rho[Z,P] \ne \rho[Z',P]$,
unless $\alpha=1$. This example illustrates the fact that the concept of law invariance on product spaces is excessively demanding.
Therefore, we do not assume law invariance of risk forms on product spaces.
\end{example}
}

%

\section{Composite disintegration}
\label{s:composite}

We now generalize our results to the product of multiple Polish spaces $\Xc_j$, $j=1,\dots,n$, with their corresponding Borel $\sigma$-algebras
$\Bc(\Xc_j)$, $j=1,\dots,n$. 
For a non\-empty set of indices $J \subset \{1,\dots,n\}$, we write
\[
\Xc_{J} = {\bigtimes_{j\in J}}\Xc_j, \quad\text{and}\quad J^c = \{1,\dots,n\} \setminus J.
\]
Let $P$ be a probability measure on $\Xc  = \displaystyle{\bigtimes}_{j=1}^n\Xc_j$.
For every { $J$ such that $J^c\ne\emptyset$,} we can
disintegrate $P$ into its marginal $P_{\Xc_J}\in \Pc(\Xc_J)$ and a transition kernel $P_{\Xc_{J^c}|\Xc_J}:\Xc_J\to \Pc(\Xc_{J^c})$
as follows:
\[
P(dx_J, dx_{J^c}) =  P_{\Xc_J}(dx_J) \, P_{\Xc_{J^c}|\Xc_J}(dx_{J^c}|x_J).
\]
For the case $J^c=\emptyset$, trivially $P(dx_J)= P_\Xc(dx_J)$.
We denote the set of transition kernels from $\Xc_J$ to $\Pc(\Xc_{J^c})$ by $\Qc_{\Xc_{J^c}|\Xc_J}$.

Exactly as in Section \ref{s:disintegration}, a monotonic, translation equivariant, and normalized risk form $\rho:\Bb(\Xc)\times\Pc(\Xc)\to \Rb$
on the product space induces a family of conditional risk operators
$\rho_{\Xc_{J^c}|\Xc_J}:\Bb(\Xc) \times \Qc_{\Xc_{J^c}|\Xc_J}\to\Bb(\Xc_J)$,
 as follows:
\begin{equation}
\label{rhoXJ}
\rho_{\Xc_{J^c}|\Xc_J}[Z,Q](x_J) = \rho[ Z, \delta_{x_J} \smallotimes Q],\quad x_J \in \Xc_J.
\end{equation}
If the risk form $\rho$ has the support property, then for each $x_J\in \Xc_J$ the value of \eqref{cond-operator} depends only on
the function $Z(x_J,\cdot)\in \Bb(\Xc_{J^c}) $ and the measure $Q(x_J) \in \Pc(\Xc_{J^c})$.
As in Section \ref{s:disintegration}, we define the conditional risk forms
functionals $\rho_{\Xc_{J^c}|x_J}:\Bb(\Xc_{J^c})\times\Pc(\Xc_{J^c})\to\Rb$,
 $x\in \Xc$, as follows:
\begin{equation}
\label{rhoxJ}
\rho_{\Xc_{J^c}|x_J}[Z(x_J,\cdot),Q(x_J)] = \rho_{\Xc_{J^c}|\Xc_J}[Z,Q](x_J),\quad x_J \in \Xc_J.
\end{equation}

\begin{definition}
\label{d:strong}
A risk form $\rho:\Bb(\Xc)\times\Pc(\Xc)\to \Rb$ is \emph{conditionally consistent with respect to $\Xc_{J}$, where $\emptyset \ne J \subset \{1,\dots,n\}$},
if for all $Z,Z' \in \Bc(\Xc)$ and all $Q,Q' \in \Qc_{\Xc_{J^c}|\Xc_J}$ the  inequality
\[
\rho_{\Xc_{J^c}|\Xc_J}[Z,Q] \le \rho_{\Xc_{J^c}|\Xc_J}[Z',Q'],
\]
implies that
\[
\rho[ Z, \lambda\smallotimes Q] \le \rho[ Z', \lambda\smallotimes Q'], \quad \forall \, \lambda \in \Pc(\Xc_J).
\]
\end{definition}


The following corollary results directly from Theorem \ref{t:total-risk}.  

\begin{corollary}
\label{c:total-risk-J}
If a risk form $\rho:\Bb(\Xc)\times\Pc(\Xc)\to \Rb$ is monotonic, normalized, translation equivariant, and conditionally consistent with respect to $\Xc_{J}$,
where $\emptyset \ne J \subset \{1,\dots,n\}$,
then
a risk form $\rho_{\Xc_J}:\Bb(\Xc_J) \times \Pc(\Xc_J)\to \Rb$ exists, such that for all $[Z,P]\in \Bb(\Xc)\times\Pc(\Xc)$
the following formula holds:
\begin{equation}
\label{total-risk-J}
\rho[Z,P] = \rho_{\Xc_J}\big[\rho_{\Xc_{J^c}|\Xc_J}[Z,P_{\Xc_{J^c}|\Xc_J}], P_{\Xc_J}\big],
\end{equation}
where the marginal risk form $\rho_{\Xc_J}$ is uniquely defined by the equation \eqref{r-def} with $(\Xc_j,\Xc_{J^c})$ replacing $(\Xc,\Yc)$.
It is monotonic, normalized, translation equivariant, and has the support property.
\end{corollary}
A question arises what is the relation between the marginal and conditional risk forms for different subspaces.

\begin{theorem}
\label{t:inheritance}
Suppose $\emptyset \neq J \subset L \subset \{1,\dots,n\}$.
If a risk form $\rho:\Bb(\Xc)\times\Pc(\Xc)\to \Rb$ is monotonic, normalized, translation equivariant, and conditionally consistent with respect to
both $\Xc_J$ and $\Xc_{L}$, then the following statements are true:
\begin{tightlist}{(iii)}
\item
For all $x_J\in \Xc_J$ the conditional risk forms $\rho_{\Xc_{J^c}|x_J}$ are monotonic, normalized, translation equivariant, and conditionally consistent with respect to
$\Xc_{L\setminus J}$;
\item For all  $x_{L}\in \Xc_{L}$ we have
$\big(\rho_{\Xc_{J^c}|x_J}\big)_{\Xc_{L^c}|x_{L\setminus J}} = \rho_{\Xc_{L^c}|x_{L}}$;
\item
The marginal risk form $\rho_{\Xc_L}$  is monotonic, normalized, translation equivariant, and {conditionally consistent with respect to $\Xc_J$};
\item We have $\rho_{\Xc_J} = \big(\rho_{\Xc_L}\big)_{\Xc_J}$.
 \end{tightlist}
 \end{theorem}
 \begin{proof}
 The monotonicity, normalization, and translation equivariance of the marginal and conditional forms follow from
 the corresponding properties of $\rho$ via formula \eqref{total-risk-J} by considering special classes of functions
in $\Bb(\Xc)$: the functions that depend only on $x_J$ (for the marginal), and the functions that depend only on
 $x_{J^c}$ (for the conditionals). The proof is identical to
 the last part of the proof of Theorem \ref{t:total-risk}.

It remains to prove conditional consistency of the conditional and marginal forms and the tower formulae (ii) and (iv).
For a fixed $x_J\in \Xc_J$, we  verify Definition \ref{d:strong} for the conditional risk  form $\rho_{\Xc_{J^c}|x_J}$
with respect to
$\Xc_{L\setminus J}$.
Let $K=L\setminus J$, and let $Z,Z' \in \Bb(\Xc_{J^c})$ and $Q,Q' \in \Qc_{\Xc_{L^c}|\Xc_K}$. Suppose
\[
\big(\rho_{\Xc_{J^c}|x_J}\big)_{\Xc_{L^c}|\Xc_K}[ Z, Q] \le
\big(\rho_{\Xc_{J^c}|x_J}\big)_{\Xc_{L^c}|\Xc_K}[ Z', Q'],
\]
which means that
\begin{equation}
\label{premise-cond}
\rho_{\Xc_{J^c}|x_J}[ Z, \delta_{x_{K}}\smallotimes Q] \le \rho_{\Xc_{J^c}|x_J}[ Z', \delta_{x_{K}}\smallotimes Q'], \quad \forall \, x_K\in \Xc_K.
\end{equation}
We can formally extend the functions $Z$ and $Z'$ to the entire domain $\Xc$ by setting
$\bar{Z}(x_J, x_{J^c}) = Z(x_{J^c})$ and $\bar{Z'}(x_J, x_{J^c}) = Z'(x_{J^c})$. We can also define
the kernels $\bar{Q}$ and $\bar{Q}'$ in $\Qc_{\Xc_{L^c}|\Xc_{L}}$ by setting
$\bar{Q}(x_J,x_K) = Q(x_K)$ and $\bar{Q}'(x_J,x_K) = Q'(x_K)$.
Then
\[
\rho_{\Xc_{J^c}|x_J}[ Z, \delta_{x_{K}}\smallotimes Q] = \rho[ \bar{Z}, \delta_{x_{L}}\smallotimes \bar{Q}].
\]
A similar equation is true for $Z'$ and $Q'$. Then \eqref{premise-cond} can be written as follows:
\[
\rho[ \bar{Z}, \delta_{x_{L}}\smallotimes \bar{Q}] \le \rho[ \bar{Z}', \delta_{x_{L}}\smallotimes \bar{Q}'].
\]
By the conditional consistency of $\rho$,
\[
\rho[ \bar{Z}, \psi\smallotimes \bar{Q}] \le \rho[ \bar{Z}', \psi\smallotimes \bar{Q}'],\quad \forall \psi\in \Pc(\Xc_J).
\]
 Let $\lambda \in \Pc(\Xc_K)$.
By setting $\psi = \delta_{x_J}\smallotimes \lambda$ in the last displayed inequality, and using the fact that $\bar{Z}$, $\bar{Z}'$, $\bar{Q}$, and $\bar{Q}'$ do not depend on $x_J$, we conclude that
\[
\rho_{\Xc_{J^c}|x_J}[ Z, \lambda \smallotimes Q] \le \rho_{\Xc_{J^c}|x_J}[ Z', \lambda \smallotimes Q'].
\]
This proves the conditional consistency of the conditional risk  form $\rho_{\Xc_{J^c}|x_J}$.

To verify  (ii), consider $f\in \Bb\big(\Xc_{L^c}\big)$, $Q\in \Qc_{\Xc_{L^c}|\Xc_{L}}$, and the natural extension $\bar{f}$ of $f$ to the entire space $\Xc$, { defined by  $\bar{f}(x_{L}, x_{L^c})= f(x_{L^c}) $}. We obtain the chain of equalities:
\begin{multline*}
\rho_{\Xc_{L^c}|x_{L}}[f,Q] = \rho\big[ \bar{f}, \delta_{x_{L}}\smallotimes Q\big]
=\rho\big[ \bar{f}, \delta_{x_J} \smallotimes \delta_{x_K}\smallotimes Q\big]\\
=
\rho_{\Xc_{J^c}|x_J}\big[\bar{f}, \delta_{x_K}\smallotimes Q\big]=
\big(\rho_{\Xc_{J^c}|x_J}\big)_{\Xc_{L^c}|x_{K}}[f,Q].
\end{multline*}

Consider now the marginal risk form $\rho_{\Xc_L}$. Let $K = L \setminus J$ and $Q,Q' \in \Qc_{\Xc_{K}|\Xc_J}$. Suppose
\begin{equation}
\label{premise-marg}
\rho_{\Xc_L}[ Z, \delta_{x_{J}}\smallotimes Q] \le \rho_{\Xc_L}[ Z', \delta_{x_{J}}\smallotimes Q'], \quad \forall \, x_J\in \Xc_J,
\end{equation}
where $Z,Z' \in \Bb(\Xc_{L})$.
We can formally extend the functions $Z$ and $Z'$ to the entire domain $\Xc$ by setting
$\bar{Z}(x_L, x_{L^c}) = Z(x_{L})$ and $\bar{Z'}(x_L, x_{L^c}) = Z'(x_{L})$. We can also define
the kernels $\bar{Q}$ and $\bar{Q}'$ in $\Qc_{\Xc_{J^c}|\Xc_{J}}$ by setting
$\bar{Q}= Q\times M$ and $\bar{Q}'= Q'\times M$, where $M$ is an arbitrary kernel in $\Qc_{\Xc_{L^c}|\Xc_{J}}$. Since the functions
$\bar{Z}$ and $\bar{Z}'$ do not depend on $x_{L^c}$, the relations \eqref{premise-marg} can be equivalently written as
\[
\rho[ \bar{Z}, \delta_{x_{J}}\smallotimes \bar{Q}] \le \rho[ \bar{Z}', \delta_{x_{J}}\smallotimes \bar{Q}'], \quad \forall \, x_J\in \Xc_J.
\]
By the conditional consistency of $\rho$ with respect to $\Xc_J$,
\[
\rho[ \bar{Z}, \lambda\smallotimes \bar{Q}] \le \rho[ \bar{Z}', \lambda\smallotimes \bar{Q}'],\quad \forall \lambda\in \Pc(\Xc_J).
\]
Since $\bar{Z}$ and $\bar{Z}'$ do not depend on $x_{L^c}$, we conclude that
\[
\rho_{\Xc_{L}}[ Z, \lambda \smallotimes Q] \le \rho_{\Xc_L}[ Z', \lambda \smallotimes Q'],
\]
which proves the conditional consistency of the marginal risk  form $\rho_{\Xc_L}$.

It remains to verify  the tower property (iv). For  $f\in \Bb\big(\Xc_{J}\big)$, we define
\[
\underline{f}(x_J,x_{J^c})
=\underline{\bar{f}}(x_L)= f(x_J).
\]
From \eqref{r-def} and the tower property for marginal measures, we obtain the chain of equalities:
\[
\rho_{\Xc_J}\big[f,P_{\Xc_J}\big] = \rho[\underline{f},P] = \rho_{\Xc_L}\big[\underline{\bar{f}},P_{\Xc_L}\big]
 =
\big(\rho_{\Xc_L}\big)_{\Xc_J}\big[f,\big(P_{\Xc_L}\big)_{\Xc_J}\big]=\big(\rho_{\Xc_L}\big)_{\Xc_J}\big[f,P_{\Xc_J}\big],
\]
which is (iv).
 \qed
 \end{proof}

\section{Risk in Two-Stage Partially Observable Systems}
\label{s:two-stage-PO}

\subsection{Fixed Observation Distribution}

Let us start from the following simple setting.
A random vector $(X,Y)$ is distributed in the product of Polish spaces $\Xc \times \Yc$ according to a measure $P$.
For a bounded measurable function $c:\Xc\times \Yc \to \Rb$, we
can evaluate the risk of $c(X,Y)$ by a risk form $\rho[c,P]$.

However, we know that we shall be able to observe the value of $X$.
After $X$ is observed, we might refine our risk evaluation of $c(X,Y)$.
Thus, a question arises: if a future possibility to observe $X$
exists, what should be our present evaluation of the risk of $c(X,Y)$, \emph{before} $X$ is observed.
Note that $Y$ is never observed.

Let us start with the problem of risk evaluation {after} $X$ is observed.
We can disintegrate $P$ into its marginal $P_{\Xc}$ on $\Xc$ and a transition kernel $P_{\Yc|\Xc}$ from $\Xc$ to $\Pc(\Yc)$:
\[
P (dx, dy) =  P_{\Xc}(dx) P_{\Yc|\Xc}(dy|x) .
\]
Suppose the risk form $\rho$ is monotonic, normalized, has the translation property and the support property.
Then the correct evaluation of the risk after $X=x$ is observed
is
\[
\rho_{\Yc|x}\big[c(x,\cdot), P_{\Yc|\Xc}(x) \big] =  \rho\big[ c, \delta_x \smallotimes P_{\Yc|\Xc}\big].
\]
This is nothing else, but the conditional risk form defined in \eqref{rhox}. As a function of $x$,
we obtain the conditional risk operator $\rho_{\Yc|\Xc}\big[c, P_{\Yc|\Xc}\big]$.
Now, to evaluate the overall risk, we calculate
$\rho_{\Xc}\Big[\rho_{\Yc|\Xc}\big[c, P_{\Yc|\Xc}\big], P_{\Xc} \Big]$. We thus arrive to the following conclusion from
Theorem \ref{t:total-risk}.
If the risk form $\rho:\Bb(\Xc\times\Yc)\times\Pc(\Xc\times\Yc)\to \Rb$ is normalized, translation equivariant, {conditionally consistent}, and has the support property, then
\[
\rho[c,P] = \rho_{\Xc}\Big[ \rho_{\Yc|\Xc}\big[c, P_{\Yc|\Xc} \big], P_{\Xc} \Big].
\]
It follows that the two risk evaluations: without and with the perspective of inspection, are identical. The mere {existence} of
inspection does not affect risk.

When a possibility of control exists, the situation is different. Suppose there are two Polish spaces
$\Uc_1$ and $\Uc_2$, which we call \emph{control spaces}. At stage one, a control $u\in U_1 \subset \Uc_1$ is chosen, where $U_1$
is a subset of $\Uc_1$.
Then an observation
of $X$ is made.  After observing the value of
$X$, we choose control $u_2\in U_2(X,u_1) \subset\Uc_2$ to minimize the risk of $c(X,Y,u_1,u_2)$. The risk is measured by the form $\rho[\cdot,\cdot]$. Here $U_2: \Xc \times \Uc_1 \rightrightarrows \Uc_2$ is a measurable multifunction representing the feasible set at the second stage.
We shall use the symbol $\pi(\cdot)\lessdot U_2(\cdot, u_1)$ to indicate that the function $\pi$ is a measurable selection of  $U_2(\cdot, u_1)$.

We may look at this problem from two perspectives. Let us start from the functional perspective. Since $u_2$ can be chosen
after $X$ is observed, we may represent it as a measurable function: $u_2 = \pi(x)$, $x\in \Xc$. Therefore, the overall
cost has the form:
\begin{equation}
\label{Zupi}
Z^{u_1,\pi}(x,y) = c(x,y,u_1,\pi(x)), \quad (x,y)\in \Xc \times \Yc.
\end{equation}
The distribution of $(X,Y)$ is $P$. The problem takes on the form
\begin{equation}
\label{functional-simple}
\begin{aligned}
\min_{u_1,\pi}&\ \rho\big[ Z^{u_1,\pi}, P\big]\\
\text{s.t.} &\  u_1 \in U_1,\\
&\ \pi(\cdot) \lessdot U_2(\cdot, u_1).
\end{aligned}
\end{equation}

We now derive a two-stage representation of this problem.
\begin{theorem}
\label{t:two-stage-simple}
We assume the following:
\begin{tightlist}{(iii)}
\item
The risk form $\rho$ is  monotonic, normalized, translation equivariant, has the support property, and is {conditionally consistent};
\item
The multifunction $U_2$ is upper-semicontinuous and has nonempty and compact values;
\item
The function $c$ is uniformly bounded, measurable, and lower-semicontinuous with respect to its second argument.
\end{tightlist}
Then problem \eqref{functional-simple} is equivalent to the two-stage problem:
\begin{equation}
\label{first-stage-problem}
\min_{u_1\in U_1} \rho_{\Xc}\big[ V(\cdot,u_1),P_{\Xc}\big],
\end{equation}
where $V(\cdot,\cdot)$ is the optimal value of the following second stage problem:
\begin{equation}
\label{second-stage-problem}
V(x,u_1) = \min_{u_2\in U_2(x,u_1)}\rho_{\Yc|x}\big[c(x,\cdot,u_1,u_2),  P_{\Yc|\Xc}(x)\big],\quad x\in \Xc,\quad u_1 \in U_1. 
\end{equation}
\end{theorem}
\begin{proof}
Since $\rho[\cdot,\cdot]$ satisfies the assumptions of Theorem \ref{t:total-risk}, the risk of the function $Z^{u_1,\pi}$
can be calculated by the risk disintegration formula:
\begin{multline*}
\label{two-stage-functional}
\rho\big[ Z^{u_1,\pi}, P\big]
= \rho_{\Xc}\Big[ \rho_{\Yc|\Xc}\big[Z^{u_1,\pi}, P_{\Yc|\Xc} \big], P_{\Xc} \Big]\\
= \rho_{\Xc}\Big[ x \mapsto \rho_{\Yc|x}\big[c(x,\cdot,u_1,\pi(x)), P_{\Yc|\Xc}(x) \big], P_{\Xc} \Big].
\end{multline*}
Then problem \eqref{functional-simple} takes on the form:
\[
\min_{u_1,\pi} \rho_{\Xc}\Big[ x \mapsto \rho_{\Yc|x}\big[c(x,\cdot,u_1,\pi(x)), P_{\Yc|\Xc}(x) \big], P_{\Xc} \Big],
\]
subject to the same constraints.
Due to the monotonicity of the marginal risk form $\rho_{\Xc}$, the smaller the values of the function
$x \mapsto \rho_{\Yc|x}\big[c(x,\cdot,u_1,\pi(x))$, the smaller the value of $\rho_{\Xc}$. Since $u_2 = \pi(x)$ may depend on $x$,
we may carry out the minimization with respect to $u_2$ inside the argument of $\rho_{\Xc}$:
\begin{multline*}
\inf_{u_1\in U_1}\inf_{\pi(\cdot)\lessdot U_2(\cdot,u_1)} \rho_{\Xc}\Big[ x \mapsto \rho_{\Yc|x}\big[c(x,\cdot,u_1,\pi(x)), P_{\Yc|\Xc}(x) \big], P_{\Xc} \Big]\\
=\inf_{u_1\in U_1} \rho_{\Xc}\Big[ x \mapsto \inf_{u_2\in U_2(x,u_1)}\rho_{\Yc|x}\big[c(x,\cdot,u_1,u_2), P_{\Yc|\Xc}(x) \big], P_{\Xc} \Big].
\end{multline*}
The only condition for the validity of this transformation is the measurability and boundedness of the optimal value
function \eqref{second-stage-problem}. This follows from Berge's theorem (see, \emph{e.g.},  \cite[Th. 1.4.16]{aubin2009set}), which can be applied due to the assumptions
(ii) and (iii). In fact, they also guarantee that the optimal value function is lower semicontinuous with respect to
$u_1$.

We conclude that problem \eqref{functional-simple} reduces to the marginal risk optimization \eqref{first-stage-problem}.
\qed
\end{proof}
Theorem \ref{t:two-stage-simple}  provides us with the second perspective on the problem. It has a hierarchical structure, similar to its expected-value full information counterpart:
after $X=x$ is observed, the second stage problem \eqref{second-stage-problem} is to minimize the conditional risk.
Then, the first stage problem takes on the form of minimizing the marginal risk \eqref{first-stage-problem} of the
second-stage optimal value. The most important conclusion
is that the risk disintegration formula allows us to write the overall problem in a hierarchical structure. The extended two-stage
risk-averse model, which is introduced and analyzed in \cite{miller2011risk} (see also \cite{gulten2015two}) is a special case of this problem.

\subsection{Controlled Observation Distribution}

Now we consider a more complex situation. There are still two Polish control spaces
$\Uc_1$ and $\Uc_2$. However,  after a control $u_1\in U_1 \subset \Uc_1$ is chosen, the distribution of the observation
$X$ depends on $u_1$. The dependence is described by a controlled kernel $K:\Yc \times \Uc_1 \to \Pc(\Xc)$. After observing
$X$, we choose control $u_2\in U_2(X,u_1) \subset\Uc_2$ to minimize the risk of $c(X,Y,u_1,u_2)$. The risk is measured by the form $\rho[\cdot,\cdot]$.

Assume the same conditions on $U_1$, $U_2$ and $\rho$ as in the previous subsection.
Let $P_Y$ be the marginal distribution of $Y$. After the first decision $u_1$ will be chosen, the joint distribution of $(Y,X)$
will become
\[
M(u_1) = P_{\Yc} \smallotimes K(\cdot,u_1),
\]
that is, $M(dy,dx|u_1) = P_{\Yc}(dy)K(dx|y,u_1)$. Therefore, denoting the second stage decision by $u_2=\pi(x)$ (it may depend on $x$),
our problem is to find
\begin{equation}
\label{Bayes-full}
\begin{aligned}
\min_{u_1,\pi}&\  \rho\big[ Z^{u_1,\pi},M(u_1)\big],\\
\text{s.t.} &\  u_1 \in U_1,\\
&\ \pi(\cdot) \lessdot U_2(\cdot, u_1).
\end{aligned}
\end{equation}
where the function $Z^{u_1,\pi}(\cdot,\cdot)$ is given by \eqref{Zupi}.

Let us develop a two-stage version of the functional problem \eqref{Bayes-full}.
The marginal distribution of the observation result is
\[
M_{\Xc}(u_1) = \int_\Yc K(y,u_1)\;P_{\Yc}(dy)
\]
where the integral is understood in the weak sense. Since the space $\Yc$ is standard, the measure $M(u_1)$ can be disintegrated
into the marginal $M_{\Xc}(u_1)$ and a transition kernel $\varGamma: \Xc \times \Uc_1 \to \Pc(\Yc)$ as follows
\[
M(u_1) = M_{\Xc}(u_1) \smallotimes \varGamma(u_1),
\]
which reads $M(dx,dy|u_1) = M_{\Xc}(dx|u_1) \varGamma(dy|x,u_1)$. The transition kernel $\varGamma$ is called the \emph{Bayes operator}.
\begin{example}
\label{e:densities}
	Assume that the joint distribution $M(u_1)$ of $(X,Y)$ has a density $q(\cdot, \cdot \mid u_1)$ with respect to
a finite product measure  $\mu_{\Xc} \otimes \mu_{\Yc}$ on $\Xc \times \Yc$. Then the Bayes operator has the form
\[
\varGamma(A|x,u_1)  = \frac{\int_A \int_\Yc q(x', y' \mid u_1)\; M_{\Yc}(dy) \; \mu_{\Yc}(dy')}
{\int_\Yc \int_\Yc q(x', y' \mid u_1)\; M_{\Yc}(dy) \; \mu_{\Yc}(dy')},\quad \forall\, A \in \Bc(\Yc).
\]
If the formula above has a zero denominator, we can formally define $\varGamma(x,u_1)$ to be an arbitrarily selected distribution on $\Yc$. \hfill $\Box$
\end{example}

With the use of the Bayes operator, we can equivalently write problem \eqref{Bayes-full} as a two-stage problem.

\begin{theorem}
\label{t:two-stage-Bayes}
We assume the following:
\begin{tightlist}{(iii)}
\item
The risk form $\rho$ is  monotonic, normalized, translation equivariant, {conditionally consistent}, and has the
support property;
\item
The multifunction $U_2$ is upper-semicontinuous and has nonempty and compact values;
\item
The function $c$ is uniformly bounded, measurable, and lower-semicontinuous with respect to its second argument.
\end{tightlist}
Then problem \eqref{Bayes-full} is equivalent to the two-stage problem:
\[
\min_{u_1\in U_1} \rho_{\Xc}\big[ V(\cdot,u_1),P_{\Xc}\big],
\]
where $V(\cdot,\cdot)$ is the optimal value of the following second stage problem:
\[
V(x,u_1) = \min_{u_2\in U_2(x,u_1)}\rho_{\Yc|x}\big[c(x,\cdot,u_1,u_2), \varGamma(x,u_1) \big],\quad x\in \Xc,\quad u_1 \in U_1.
\]
\end{theorem}
\begin{proof}
With the use of the Bayes formula, we can write problem \eqref{Bayes-full} as follows:
\[
\min_{u_1\in U_1} \min_{\pi(\cdot)\lessdot U_2(\cdot,u_1)} \rho\big[ Z^{u_1,\pi},P_{\Xc}(u_1) \smallotimes \varGamma(u_1)\big].
\]
Since the risk form $\rho$ satisfies the assumptions of Theorem \ref{t:total-risk}, we can disintegrate it to obtain
the following equivalent form:
\[
\min_{u_1\in U_1} \min_{\pi(\cdot)\lessdot U_2(\cdot,u_1)}
 \rho_{\Xc}\Big[ x \mapsto \rho_{\Yc|x}\big[c(x,\cdot,u_1,\pi(x)), \varGamma(x,u_1) \big], P_{\Xc}(u_1) \Big].
\]
The remaining considerations are the same as in the proof of Theorem \ref{t:two-stage-simple}.
Due to the monotonicity of the marginal risk form $\rho_{\Xc}$, the smaller the values of the function
$x \mapsto \rho_{\Yc|x}\big[c(x,\cdot,u_1,\pi(x)), \varGamma(x,u_1) \big]$, the smaller the value of $\rho_{\Xc}$. Since $u_2 = \pi(x)$ may depend on $x$,
we may carry out the minimization with respect to $u_2$ inside the argument of $\rho_{\Xc}$:
\begin{multline*}
\inf_{u_1\in U_1} \inf_{\pi(\cdot)\lessdot U_2(\cdot,u_1)}\rho_{\Xc}\Big[ x \mapsto \rho_{\Yc|x}\big[c(x,\cdot,u_1,\pi(x)), \varGamma(x,u_1) \big], P_{\Xc}(u_1) \Big]\\
=\inf_{u_1\in U_1} \rho_{\Xc}\Big[ x \mapsto \inf_{u_2\in U_2(x,u_1)}\rho_{\Yc|x}\big[c(x,\cdot,u_1,u_2), \varGamma(x,u_1) \big], P_{\Xc}(u_1) \Big].
\end{multline*}
The only difference is that the new marginal distribution and the Bayes operator are the disintegration components of the
probability measure and feature in the risk disintegration formula. The ``$\inf$'' operation in the second stage problem can be replaced by ``$\min$'' because of  conditions (ii) and (iii).
\qed
\end{proof}

{

\section{Conclusions}

Our work initiates systematic research of risk measures considered as functionals of two arguments: a function on a space $\Xc$ and an underlying probability measure
on the $\sigma$-algebra of Borel subsets of $\Xc$.
Such functionals, which we call risk forms, occur in two- and multi-stage optimization models, in which the probability measure depends on decisions,
and in models, in which only partial observation is available and Bayesian updates of the probability measure are employed.

Two main results: the Kusuoka representation and the risk disintegration formula, generalize the classical properties of risk measures to the new setting.
The derivation of the dual and Kusuoka representations hinges on novel duality theory for functionals of quantile functions \cite{dentcheva2014risk}.
The risk disintegration formula uses a new concept of conditional consistency.

For both contributions, essential is the boundedness of the functions under consideration. It allows us to consider arbitrary probability measures
on the spaces involved. In the first group of results, duality between bounded functions and finitely additive measures plays a role; consistency with the increasing convex order allows to pass to countably additive measures. In the second group of results, the very concept of conditional consistency uses Dirac delta measures
(in products with conditional measures), which are natural in a space of bounded functions.

The theory of risk measures is well-developed in the spaces of integrable functions. A fixed underlying probability measure is essential
for defining the space of functions. A fundamental challenge is to extend the theory of risk forms to important classes of unbounded functions.
It may require to precisely define both domains of the risk form: a broader class of functions and a narrower class of probability measures, so that
properties of the forms can be preserved. In particular, one could conjecture that a generalized Kusuoka representation could be developed in such a broader
setting, but this requires further research.

}

\bibliographystyle{plain}

\end{document}